\documentclass[11pt,a4paper]{article}

\usepackage{amsmath,amssymb,latexsym}
\usepackage{mathtools}
\usepackage{graphicx}
\usepackage{epsfig}
\usepackage{longtable}
\usepackage{amsthm}
\usepackage{xcolor}
\usepackage{multirow}
\usepackage{caption}
\usepackage{subcaption}
\usepackage{ upgreek }
\usepackage{amssymb}
\usepackage{dsfont}
\usepackage{mathrsfs}

\usepackage{lineno}

\newtheorem{theorem}{Theorem}

\newtheorem{definition}{Definition}
\newtheorem{example}{Example}
\newtheorem{remark}{Remark}

\def\XX{{\mathbf X}}

\def\xx{{\mathbf x}}
\def\yy{{\mathbf y}}
\def\uu{{\mathbf u}}
\def\aa{{\mathbf a}}
\def\bb{{\mathbf b}}

\begin{document}

\title{ \Large\bf Convex lineability in copula and quasi-copula sets}
\author{Enrique de Amo$^{\rm a}$, Juan Fernández-Sánchez$^{\rm b}$, David García-Fernández$^{\rm b}$,\\ Manuel Úbeda-Flores$^{\rm a,b,}$\footnote{Corresponding author}\\
\small{$^{\rm a}$Department of Mathematics, University of Almería, 04120 Almería, Spain}\\
\small{\texttt{edeamo@ual.es,\,\,\,mubeda@ual.es}}\\
\small{$^{\rm b}$Research Group of Theory of Copulas and Applications, University of Almería, 04120 Almería, Spain}\\
\small{\texttt{juanfernandez@ual.es,\,\,\,dgf992@inlumine.ual.es}}\\
}
\maketitle

\begin{abstract}
 In this paper, we investigate several subsets of $n$-copulas and $n$-quasi-copulas from the perspective of convex-lineability and the recently introduced concept of convex-spaceability. Our purpose is to determine when such families contain extremely large algebraic structures, namely linearly independent sets of cardinality of the continuum whose convex hull, and in some cases a closed convex linearly independent subset, remain entirely inside the class under study. These include the families of asymmetric copulas, copulas with maximal asymmetric measure, and proper $n$-quasi-copulas, among others. In contrast, for several other natural classes of copulas we show that (maximal) convex lineability holds while convex spaceability remains an open problem.
\end{abstract}
\bigskip

\medskip


\section{Introduction}\label{sec:copulas}

The concept of lineability appears for the first time in \cite{Se06}, although it can be found implicitly in works several decades earlier. The objective is to find vector spaces in subsets of vector spaces. In the case of copulas (\cite{Durante2016book,Ne06}), it is not possible to find these vector spaces since we cannot multiply by negative numbers. For such cases, the concept of coneability was introduced (\cite{Ai08}).  This new concept is also not suitable because we cannot multiply by different numbers from $1$ without losing the boundary conditions of copulas.

Recently, the concept of convex-lineability has been introduced in \cite{Ber19}, which proves to be appropriate within the set of copulas.  This notion detects large linearly independent families whose convex hull remains inside the nonlinear set of interest. When the cardinality of the set of linearly independent vectors equals that of the ambient vector space, we refer to the set as being maximally convex lineable, representing an optimal situation. This property reveals profound algebraic and geometric richness within sets that, at first glance, appear to lack any linear structure. Convex lineability often uncovers hidden vectorial behavior in collections of pathological, irregular, or highly nonlinear objects, such as nowhere differentiable functions, singular measures, or solutions to nonlinear differential equations.

In addition to convex lineability, in this work we make systematic use of the related concept of convex spaceability, meaning the existence inside a given nonlinear set of a closed, convex, infinite linearly independent subset. This refinement captures situations in which not only algebraic largeness but also topological completeness is preserved. For some classical classes of copulas (such as asymmetric copulas or copulas with maximal asymmetry), we prove that maximal convex lineability can in fact be strengthened to maximal spaceable convex lineability, whereas in other settings convex spaceability remains unknown.

Convex lineability and convex spaceability reveal the presence of unexpectedly large algebraic and geometric configurations inside sets that are not linear and often highly irregular. Such phenomena have been observed in sets of nowhere differentiable functions, singular measures, or solutions to nonlinear differential equations. In the context of copulas, the existence of large convex lineable or convex spaceable families enhances the understanding of dependence modeling, enriches the set of available examples for extremal behavior, and provides new tools for approximation and interpolation in probabilistic models.

In this paper, we work in the space of $n$-copulas on $[0,1]^n$, endowed with the supremum norm, and study convex lineability and convex spaceability of several natural families of copulas and quasi-copulas. Our main contributions include the construction of maximally convex lineable families in a number of important classes, and the identification of specific instances where convex spaceability can also be achieved.

The structure is as follows. In Section~\ref{sec:prem} we recall the necessary preliminaries on copula theory, quasi-copulas, ordinal sums, and the notions of convex lineability and convex spaceability. Section~\ref{sec:convex} contains the main results: we establish convex lineability and spaceability properties in a variety of relevant families, including copulas with fractal support, asymmetric copulas, copulas with maximal asymmetry, Lipschitz-type copulas, proper quasi-copulas, and sequences of copulas under different topologies. Section~\ref{sec:conc} presents the conclusions and outlines several open problems.

\section{Preliminaries}\label{sec:prem}

This section introduces the foundational mathematical concepts for our subsequent analysis. We first briefly review the core principles of Copula Theory and then formalize the notion of convex lineability, the algebraic tool central to this work.

\subsection{Copula theory}

Let $n\geq 2$ be a natural number. An $n$-dimensional copula (for short, $n$-copula), $C$, is the restriction to $[0,1]^n$ of a continuous $n$-dimensional distribution whose univariate margins are uniform on $[0,1]$. The following definition is also equivalent: 

\begin{definition}\label{def:copula}
    An $n$-copula $C$ is a function $C:[0,1]^n\longrightarrow [0,1]$ satisfying the following properties:
    \begin{itemize}
        \item [i)] $C(\uu)=0$ for every $\uu:=(u_1,u_2,\ldots,u_n)\in[0,1]^n$ such that at least one coordinate of $\uu$ is equal to $0$.
        \item [ii)] $C(\uu)=u_k$ whenever all coordinates of $\uu$ are equal to $1$ except, maybe, $u_k$.
        \item [iii)] For every $\aa=(a_1,a_2,\ldots,a_n)$ and $\bb=(b_1,b_2,\ldots,b_n)$ in $[0,1]^n$, such that $a_k\leq b_k$ for all $k=1,2,\ldots,n,$
        $$V_C(B):=\sum_{\bf c} {\rm sgn}({\bf c})C({\bf c})\geq 0,$$
        where $B$ denotes the $n$-box $[\aa,\bb]=\times_{k=1}^n[a_k,b_k]$ and the sum is taken over the vertices ${\bf c}=(c_1,c_2,\ldots,c_n)$ of $[\aa,\bb]$, that is, each $c_k$ is equal to either $a_k$ or $b_k$, and ${\rm sgn}({\bf c})=1$ if $c_k=a_k$ for an even number of $k's$, and ${\rm sgn}({\bf c})=-1$ otherwise.
    \end{itemize}
\end{definition}

We will refer to $V_C$ as the {\em volume} associated with $C$ (on $n$-boxes).

The importance of copulas lies in Sklar's theorem \cite{Sk59}, which we recall now---for a complete proof of this result, see \cite{Ub2017}.

\begin{theorem}[Sklar]
    Let $\XX=(X_1,X_2,\ldots,X_n)$ be a random $n$-vector with joint distribution function $H$ and one-dimensional marginal distributions $F_1,F_2,\ldots,F_n$. Then there exists an $n$-copula, $C$, uniquely determined on $\times_{i=1}^n Range(F_i)$, such that 
    $$H(\xx)=C(F_1(x_1), F_2(x_2),\ldots,F_n(x_n))$$
    for all $\xx\in[-\infty,+\infty]^n$. If all the marginals $F_i$ are continuous then the $n$-copula $C$ is unique.
\end{theorem}

The copula for independent random variables is the {\it product} (or {\it independence}) {\it copula}, defined as
$\Pi^n({\bf u}) := \prod_{i=1}^{n}u_i$ for all ${\bf u}$ in $[0,1]^n$.

Let $\mathcal{C}^n$ denote the set of all $n$-copulas.

Alsina {\it et al.} \cite{Alsina1993} introduced the notion of a {\it quasi-copula}---a more general concept to that of a copula---in order to characterize operations on distribution functions that can, or cannot, be derived from operations on random variables defined on the same probability space. In the last years these functions have attracted an increasing interest by researchers in some topics of fuzzy sets theory, such as preference modeling, similarities and fuzzy logics. For a complete overview of quasi-copulas, we refer to \cite{Arias2020,Sempi2017,Sempi2023}.

Cuculescu and Theodorescu \cite{Cuculescu2001} characterize a multivariate quasi-copula, or $n$-quasi-copula, ---a 2-quasi-copula was previously characterized in \cite{Ge99}--- as a function $Q\colon [0,1]^{n}\to [0,1]$ that satisfies conditions i) and ii) in Definition \ref{def:copula}, but instead of iii), the weaker conditions:
\begin{itemize}
\item[{\it iv)}] $Q$ is increasing in each variable; and
\item[{\it v)}]  the {\it Lipschitz condition}: $\vert Q(\uu)-Q({\bf v})\vert \le \sum_{i=1}^{n}\vert u_{i}-v_{i}\vert$, for any ${\bf u},{\bf v}\in [0,1]^{n}$.
\end{itemize}

Let $\mathcal{Q}^n$ denote the set of all $n$-quasi-copulas.

As with $n$-copulas, we will refer to $V_Q(B)$ as the $Q$-volume of an $n$-quasi-copula on an $n$-box $B$, and which is not necessarily positive.

Every $n$-quasi-copula $Q$ satisfies the following inequalities:
$$W^{n}(\uu):= \max\left\{\sum_{i=1}^{n}u_i-n+1,0\right\}\le Q(\uu)\leq \min\{\uu\}=:M^n(\uu)$$
for all ${\bf u}\in [0,1]^n$. $M^n$ is an $n$-copula for every $n\ge 2$, $W^{2}$ (or simply, $W$) is a 2-copula and $W^n$, for $n\ge 3$, is a {\it proper} $n$-quasi-copula, i.e., an $n$-quasi-copula but not an $n$-copula.

Let $J$ be a finite or countable subset of the natural numbers $\mathbb{N}$. Let $\{\,]a_k,b_k[\,\}_{k\in J}$ be a family of subintervals of the unit interval $[0,1]$ indexed by $J$, 
and let $\{Q_k\}_{k\in J}$ be a family of $n$-quasi-copulas also indexed by $J$. 
It is required that any two intervals $]a_k,b_k[$, $k\in J$, have at most one endpoint in common. The {\it ordinal sum} $Q$ of the family $\{Q_k\}_{k\in J}$ with respect to the family of intervals $\{\,]a_k,b_k[\,\}_{k\in J}$ 
is defined, for all $\mathbf{u}\in[0,1]^n$, by
\[
Q(\mathbf{u}) :=
\begin{cases}
a_k + (b_k - a_k)\,
Q_k\!\left(
\displaystyle
\frac{\min\{u_1,b_k\}-a_k}{b_k-a_k},
\ldots,
\frac{\min\{u_n,b_k\}-a_k}{b_k-a_k}
\right),
\\[10pt] 
\hspace{3cm} \text{if } \min\{u_1,u_2,\ldots,u_n\}\in ]a_k,b_k[ \text{ for some } k\in J,\\ 
\min\{\uu\}, \quad \text{elsewhere.}
\end{cases}
\]
The ordinal sum of $n$-quasi-copulas is again an $n$-quasi-copula for all $n\ge 2$ (see \cite{Durante2020,Mesiar2010}). For such an $n$-quasi-copula $Q$, one writes $
Q = (\langle a_k,b_k,Q_k\rangle)_{k\in J}$. Another type of ordinal sum based on $W$ was introduced in \cite{Mesiar04} as
$$
Q(u,v) = 
\begin{cases} 
(b_k - a_k) \cdot C_k \left( \displaystyle\frac{u - a_k}{b_k - a_k}, \frac{v + b_k - 1}{b_k - a_k} \right), & \text{if } (u,v) \in [a_k, b_k] \times [1 - b_k, 1 - a_k], \\ 
W(u,v), & \text{otherwise},
\end{cases}$$
and it is called $W$-{\it ordinal sum}.

\subsection{Convex-lineability}
The notion of lineability that we will explore throughout this article consists, as previously mentioned in the introduction, in identifying large algebraic structures within subsets of a given topological vector space. The origin of this concept can be traced back to the study of functions of bounded variation on the interval $[0,1]$. It was shown in \cite{Lev40} that this set does not contain any infinite-dimensional closed subspace. For a more detailed treatment of this theory, the reader is referred to \cite{Aron16}. The following definition of lineability was first introduced in \cite{Aron05}.

\begin{definition}
    A subset $M$ of a vector space $X$ is said to be lineable if $M\cup\{0\}$ contains an infinite dimensional vector space. If $X$ is, in addition, a topological vector space, then $M$ is called spaceable if $M\cup\{0\}$ contains a closed infinite dimensional vector subspace.
\end{definition}

As previously mentioned, the concept of lineability is not applicable to the set of copulas. When multiplying by negative scalars or by constants different from one, the boundary conditions that every copula must satisfy are no longer preserved. For this reason, the notion of convex lineability, introduced in \cite{Ber19}, is a much more suitable framework for the study of copulas.

\begin{definition}
    Let $M$ be a subset of a vector space $X$. Then $M$ is said to be convex lineable if there exists an infinite linearly independent subset $A\subset X$ such that  ${\rm conv}(A)\subset M$, where ${\rm conv}(A)$ denotes the convex hull of a subset $A$ of the vector space $X$, that is, 
    $${\rm conv}(A)=\left\{\sum_{i=1}^n\lambda_i x_i: x_i\in A,\hspace{0.1cm} \lambda_i\in [0,1],\hspace{0.1cm} \sum_{i=1}^n\lambda_i=1,\hspace{0.1cm} n\in\mathbb{N}\right\}.$$
When the cardinal of the linear independent vectors coincides with the cardinal of the vector space $X$, we will say that $M$ is {\it maximal convex lineable}.
\end{definition}

Based on the concept of convex lineable, we present the new concept of convex spaceable. 

\begin{definition}
    Let $M$ be a subset of a vector space $X$. Then $M$ is said to be {\it convex spaceable} if there exists a convex, closed and containing an infinite linearly independent subset $A\subset M$. We say \(M\) is \textit{maximal spaceable convex lineable} if the cardinality of $A$ equals the cardinality of \(X\).
\end{definition}

Along this paper we will denote by $\mathfrak{c}$ the cardinality of continuum, i.e., the cardinality of the set of real numbers.

\section{Convex-lineability in copula and quasi-copula sets}\label{sec:convex}

The set $\mathcal{C}^n$ plays a central role in the theory of dependence modeling among random variables. It is compact, with the supremum norm, and a convex subset of all distribution functions on $[0,1]^n$, but it is neither a vector space nor closed under arbitrary linear combinations.
Convex lineability allows us to explore the internal geometry of $\mathcal{C}^n$, revealing vast families of copulas that exhibit both algebraic and convex stability. These families can span entire subspaces (of cardinality $\mathfrak{c}$) while remaining within prescribed statistical or analytical constraints. Maximal convex lineable subsets provide a dense and structured supply of examples or test functions with fine-tuned properties. This is particularly relevant in the study of copula-based function spaces, optimization problems involving dependence structures, or the generation of basis functions for simulation algorithms. From a mathematical standpoint, knowing that a certain class of copulas contains a subspace of maximal dimension indicates that this class is not “small” or “exceptional” within $\mathcal{C}^n$, but rather large, rich, and structurally complete in a precise sense.

In practical applications—such as finance, hydrology, or insurance—copulas are used to model multivariate dependence. Convex combinations of copulas often correspond to mixtures of dependence structures, allowing for richer and more adaptable models. Working within a convex lineable family ensures that all such mixtures preserve desired properties, leading to robust modeling tools.

\begin{remark}
    We recall that the metric space $(\mathcal{C}^n,d_\infty)$—i.e., the space of $n$-dimensional copulas with the sup norm—is compact: see, e.g., \cite[Theorem 1.7.7]{Durante2016book}.
\end{remark}

We now present some of the sets of interest that will be studied throughout this work, along with certain details necessary for a clear understanding of the arguments developed later in this paper.

\subsection{Copulas with fractal support}

Following \cite{Fre05}, there exist self-similar copulas whose supports---denoted by supp---are fractals of prescribed Hausdorff dimension. We recall that the Hausdorff dimension $\dim_{\mathcal H}(E)$ of a bounded subset $E\subset\mathbb{R}^2$ is defined via Hausdorff measures (see, e.g., \cite{Falconer2014}). 

A \emph{transformation matrix} is a nonnegative $n\times n$ matrix $T$ whose entries sum to $1$ and in which no row or column is entirely zero. Each such $T$ defines a transformation $T(C)$ acting on copulas by subdividing the unit square according to the entries of $T$ and redistributing the mass of $C$ proportionally (see \cite{Fre05}). For every $T$ there exists a unique invariant copula $C_T$ satisfying $T(C_T)=C_T$, whenever at least one entry equals zero, whose support is a self-affine fractal set.

For $r\in(0,1/2)$, consider the matrices
\[
T_r=\begin{bmatrix}
r/2 & 0 & r/2\\
0 & 1-2r & 0\\
r/2 & 0 & r/2
\end{bmatrix}.
\]
The invariant copula $C_{T_r}$ has fractal support with Hausdorff dimension $s\in(1,2)$ determined by
\[
4r^{\,s}+(1-2r)^{\,s}=1.
\]

We now show that the set $S_{1,s}$ of copulas whose supports have Hausdorff dimension $s$ is maximal convex lineable. To achieve this, we will draw upon the concepts developed in \cite{Edeamo12,Fre05}.

\begin{theorem}
The set $S_{1,s}$ is maximal convex lineable.
\end{theorem}

\begin{proof}
It suffices to consider the set of copulas described in \cite{Edeamo12}: for a fixed $r$,
$$\mathcal{C}_{r,a}:=\left\{C_{r,a}:r\in\left(0,\frac{1}{2}\right), 0<a<r\right\}\subseteq S_{1,s},$$
where $C_{r,a}$ is an invariant copula for the transformation matrix 
$$T_{r,a}=\begin{bmatrix}
r-a & 0 & a\\
0 & 1-2r & 0\\
a & 0 & r-a\\
\end{bmatrix}.$$
Its support is a fractal whose Hausdorff dimension is the unique solution $s\in(1,2)$ to the equation $4r^s+(1-2r)^s=1$, since it coincides with the support of $C_{T_r}$ (see \cite{Edeamo12}).

For fixed $r$ and distinct parameters $a_1\neq a_2$, the associated invariant probability measures $\mu_{r,a_1}$ and $\mu_{r,a_2}$ are mutually singular and (crucially) supported on the same compact fractal set (see \cite[ Proposition 3.24]{Edeamo12} and \cite{Fre05}). We prove that $\mathcal{C}_{r,a}$ is an infinite linearly independent family whose convex hull is contained in $S_{1,s}$.

For each copula $C$ denote by $\mu_C$ the Borel probability measure on $[0,1]^2$ determined by
$C(u,v)=\mu_C\big([0,u]\times[0,v]\big)$ for all $(u,v)\in[0,1]^2$. If a finite linear combination of members of $\mathcal{C}_{r,a}$ vanishes pointwise,
\[
\sum_{i=1}^n\lambda_i C_{r,a_i}\equiv 0,
\]
then, evaluating on each generating rectangle $[0,u]\times[0,v]$, we obtain
\[
\sum_{i=1}^n\lambda_i\mu_{r,a_i}\big([0,u]\times[0,v]\big)=0
\quad\text{for all }(u,v).
\]
By the uniqueness of measures determined by values on the generating $\pi$-system of rectangles---i.e., a family of sets that is closed under finite intersections---, the signed Borel measure $\sum_{i=1}^n\lambda_i\mu_{r,a_i}$ is identically zero. Now, since the measures $\mu_{r,a_i}$ are mutually singular, for each index $j$ there exists a Borel set $A_j$ with $\mu_{r,a_j}(A_j)=1$ and $\mu_{r,a_i}(A_j)=0$ for $i\neq j$. Evaluating the zero measure on $A_j$ yields
\[
0=\sum_{i=1}^n\lambda_i\mu_{r,a_i}(A_j)=\lambda_j\mu_{r,a_j}(A_j)=\lambda_j,
\]
hence $\lambda_j=0$. As this holds for every $j$, all coefficients vanish and the family $\mathcal{C}_{r,a}$ is linearly independent.

Let $C=\sum_{i=1}^n\lambda_i C_{r,a_i}$ be any convex combination of members of $\mathcal{C}_r$ (so $\lambda_i\ge0$, $\sum\lambda_i=1$). The corresponding measure is the convex combination of the measures,
\[
\mu_C=\sum_{i=1}^n\lambda_i\mu_{r,a_i},
\]
and since each $\mu_{r,a_i}$ is supported in the same fractal set $S$ we obtain $\mathrm{supp}(\mu_C)= S$. Therefore, the support of $C$ is $S$ and has Hausdorff dimension $s$; hence $C\in S_{1,s}$. This proves ${\rm conv}(\mathcal{C}_{r,a})\subset S_{1,s}$.

The parameter set $(0,r)$ has cardinality $\mathfrak{c}$ and produces $\mathfrak{c}$ linearly independent copulas in $\mathcal{C}_{r,a}$, so the convex-lineable dimension is $\mathfrak{c}$ and $S_{1,s}$ is maximal convex lineable, which completes the proof.
\end{proof}

\subsection{Asymmetric copulas}

Let $S_2$ denote the set of asymmetric $2$-copulas, i.e., the bivariate copulas that satisfy that $C(u,v)\neq C(v,u)$ for at least one $(u,v)\in[0,1]^2$.

Now we will show that the set of bivariate asymmetric copulas is maximal spaceable convex lineable.

\begin{theorem}\label{th:s2}
  The set $S_2$ is maximal spaceable convex lineable.
\end{theorem}

\begin{proof}
    Consider the set $L_1$ formed by the quadrilateral with vertices $(0,2/3),(0,3/4),(1/3,1)$ and $(1/4,1)$ and its interior; and $L_2$ the quadrilateral with vertices $(1/3,0), (1/4,0), (1,2/3)$ and $(1,3/4)$ along with its interior. We denote by $L=L_1\cup L_2$ and by $\mathcal{C}_L$ the set of 2-copulas whose support is contained in $L$. We show that every copula in $\mathcal{C}_L$ is asymmetric. Let $C$ be a copula in $\mathcal{C}_L$. Then there exists a rectangle $B\subset\text{int}(L)$ such that $V_C(B)>0$. Suppose, by contradiction, that it is symmetric. Then it must follow that $V_C(B_S)=V_C(B)>0$, where $B_S$ is the symmetric rectangle with respect to $B$; i.e., if $B=[u,u']\times[v,v']$, then $B_S=[v,v']\times[u,u']$. But $B_S\cap L=\emptyset$, since $B_S$ is contained in the quadrilateral with vertices $(2/3,0), (3/4,0), (1,1/4)$ and $(1,1/3)$. Hence, it is a contradiction and $C$ is asymmetric. Therefore $\mathcal{C}_L\subset S_2$ (see Figure \ref{fig:rectangles}).
    \begin{figure}
        \centering
        \includegraphics[width=0.35\linewidth]{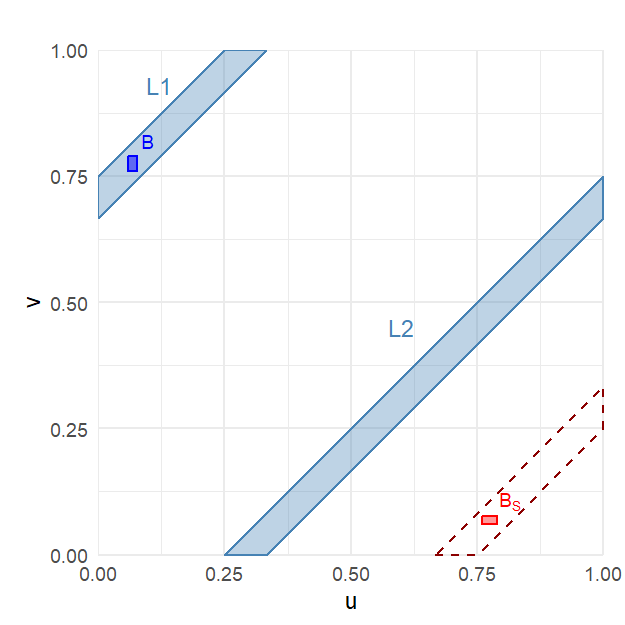}
        \caption{Situation described in the proof of Theorem \ref{th:s2}.}
        \label{fig:rectangles}
    \end{figure}

    Moreover, $\mathcal{C}_L$ is convex and closed. Firstly, we prove convexity: Let $C_1$ and $C_2$ be two copulas in $\mathcal{C}_L$, then, for $\lambda\in[0,1]$ the copula 
    $$C=\lambda C_1+(1-\lambda) C_2$$
    belongs to $\mathcal{C}_L$, since it is straightforward that $V_C=\lambda V_{C_1}+(1-\lambda)V_{C_2}$. Now, we prove that $\mathcal{C}_L$ is closed: Let $\{C_n\}_{n\in\mathbb{N}}$ be a sequence of elements in $\mathcal{C}_L$ that converges to $C$ and $B$ a rectangle such that int$(B)\cap L=\emptyset$, then $V_{C_n}(B)=0$ for all $n\in\mathbb{N}$. Thus, 
    $$V_C(B)=\lim_{n\to\infty}V_{C_n}(B)=0,$$
    that is, $C\in\mathcal{C}_L$ since its support is contained in $L$. That proves $\mathcal{C}_L$ is closed.
    
    For $\theta\in[2/3,3/4]$, let us consider the continuous family of asymmetric copulas $\mathcal{C}_\theta$ formed by
$$C_\theta(u,v)=\begin{cases}
    \min\{u,v-\theta\}, & (u,v)\in[0,1-\theta]\times[\theta,1],\\
    \min\{u+\theta-1,v\}, & (u,v)\in[1-\theta,1]\times[0,\theta],\\
    W^2(u,v), & otherwise
\end{cases}$$
(the set of their supports is $L_1\cup L_2$ shown in Figure \ref{fig:rectangles}). Each copula $C_\theta$ is a {\it shuffle of Min}  \cite{Mi1992} (or also a $W$-ordinal sum) and its probability mass is uniformly distributed on the line segments $(0,\theta), (1-\theta,1)$ and $(1-\theta,0), (1,\theta)$. Therefore, given $\theta_1,\theta_2\in[2/3,3/4]$,   $C_{\theta_1}$ and $C_{\theta_2}$ are mutually singular whenever $\theta_1\neq \theta_2$. So, the set $\{C_\theta:\theta\in[2/3,3/4]\}$ is linearly independent, and its cardinality is $\mathfrak{c}$. The support of all the copulas in that set is contained in $L$. All these results combined prove that $S_2$ is maximal spaceable convex lineable.
\end{proof}

\subsection{Copulas with maximum asymmetric measure}

According to \cite{Kle06,Ne07}, to quantify the degree of asymmetry of a 2-copula $C$, a natural and widely used measure is defined by:
$$\|d_C\|_\infty=\sup_{(u,v)\in[0,1]^2}\{ d_C(u,v)\},$$
where $d_C:[0,1]^2\longrightarrow[0,1]$ is given by $d_C(u,v)=|C(u,v)-C(v,u)|$.
This measure vanishes if, and only if, $C$ is symmetric, and reaches its maximal value on the most asymmetric copulas.

It is known that
$$\|d_C\|_\infty\leq\frac{1}{3}$$
for every copula $C$ (see \cite{Kle06,Ne07}), and equality holds whenever $C(2/3,1/3)=1/3$ and $C(1/3,2/3)=0$ or $C(1/3,2/3)=1/3$ and $C(2/3,1/3)=0$. We say that a copula $C$ has maximal asymmetry if $\|d_C\|_\infty=1/3$.

 Let $S_2'$ be the set of copulas with maximum asymmetric measure. Then we have the following result:

\begin{theorem}\label{th:S2'}
   The set $S_2'$ is maximal spaceable convex lineable.
\end{theorem}

\begin{proof}
    Given $\Phi:(\mathcal{C}^2,d_\infty)\longrightarrow(\mathcal{C}^2,d_\infty)$ as follows
    $$\Phi(C)(u,v):=\begin{cases}
        \displaystyle\frac{1}{3}C(3u,3v-2), & (u,v)\in[0,1/3]\times[2/3,1],\\
        \min\{u-1/3,v\}, & (u,v)\in[2/3,1]\times[0,2/3],\\
        W^2(u,v),& \text{otherwise}. 
    \end{cases}$$

    The mapping $\Phi$ is well defined because it is constructed as a $W$-ordinal sum, and it is an homeomorphism---since it is about a continuous map defined on a compact set $(\mathcal{C}^2,d_\infty)$ into a Hausdorff space---since 
    \begin{align*}
        d_\infty(\Phi(C_1),\Phi(C_2))&=\sup_{(u,v)\in[0,1]^2}|\Phi(C_1)(u,v)-\Phi(C_2)(u,v)|\\
        &=\sup_{(u,v)\in[0,1/3]\times[2/3,1]}\left|\frac{1}{3}C_1(3u,3v-2)-\frac{1}{3}C_2(3u,3v-2)\right|\\
        &=\frac{1}{3}\sup_{(u,v)\in[0,1]^2}|C_1(u,v)-C_2(u,v)|\\
       &=\frac{1}{3}d_\infty(C_1,C_2),
    \end{align*} 
    Hence, $\Phi$ is Lipstchiz, in particular continuous. The metric space $(\mathcal{C}^2,d_\infty)$ is compact, so $\Phi(\mathcal{C}^2)$ is compact and therefore, closed. It is clear that $\Phi(\mathcal{C}^2)\subset S_2'$, since $\Phi(C)(\frac{2}{3},\frac{1}{3})=\frac{1}{3}$ and $\Phi(C)(\frac{1}{3},\frac{2}{3})=0$, for every $C\in\mathcal{C}^2$. It is also straightforward that $\Phi(\lambda C_1+(1-\lambda)C_2)=\lambda\Phi(C_1)+(1-\lambda)\Phi(C_2)$ for $\lambda\in[0,1]$. Since the mapping $\Phi$ is a homeomorphism and preserves convex linear combinations, it follows that if $A\subset\mathcal{C}^2$ is a linearly independent family of copulas, then $\Phi(A)$ is also a linearly independent family. So $S_2'$ is maximal spaceable convex lineable.
\end{proof}

For the next result we recall that given a topological space $(X, d)$, a subset $S \subset X$ is said to be \textit{nowhere dense} in $X$ if the interior of its closure is empty.

\begin{theorem}
    The set $S'_2$ is closed and nowhere dense in $(\mathcal{C}^2,d_\infty).$
\end{theorem}

\begin{proof}
    Now we prove that $S_2'$ is closed. Define the map $\Phi:\mathcal{C}^2\longrightarrow\mathbb{R}$ as follows 
    $$\Phi(C):=\|d_C\|_\infty=\sup_{(u,v)\in[0,1]^2}|C(u,v)-C(v,u)|.$$
    We claim $\Phi$ is continuous with respect to the sup norm $d_\infty$. Indeed, for any two copulas, $C$ and $C'$, and any $(u,v),$
    $$\big||C(u,v)-C(v,u)|-|C'(u,v)-C'(v,u)|\big|\leq |C(u,v)-C'(u,v)|+|C(v,u)-C'(v,u)|\leq 2\|C-C'\|_\infty.$$
    Taking the supremum over $(u,v)$ yields
    $$|\Phi(C)-\Phi(C')|\leq2\|C-C'\|_\infty,$$
    hence $\Phi$ is Lipschitz (in particular, continuous). Therefore, the level set 
    $$S_2'=\Phi^{-1}(\{1/3\})$$
    is closed in $\mathcal{C}^2$ since it is the preimage of the closed singleton $\{1/3\}$ by the continuous map $\Phi$.  

    Finally, we show that $S_2'$ has no interior points. Suppose that $C$ is an interior point; then there exists a value $n_0$ such that $\frac{n-1}{n}C+\frac{1}{n}\Pi^2\in S_2'$ for $n>n_0;$ since $d_\infty(C,\frac{n-1}{n}C+\frac{1}{n}\Pi^2)=\frac{1}{n}d_\infty(C,\Pi^2)$. But this is not true because none of these copulas vanish on $(\frac{1}{3},\frac{2}{3})$. As $S_2'$ is closed with no interior points, it is nowhere dense.
\end{proof}

\subsection{$1$-$p$-Lipschitz copulas that are not $1$-$p'$-Lipschitz for any $p'>p$}

Let $C$ be a 2-copula satisfying
$$|C(\xx)-C(\yy)|\leq \|\xx-\yy\|_p,$$
for any $\xx=(x_1,x_2)$, $\yy=(y_1,y_2)\in[0,1]^2$, where $\|\xx\|_p=(\sum_{i=1}^2|x_i|^p)^{1/p}$. For a comprehensive study of this type of copulas and their properties, the reader is referred to \cite{DeB10}. Note that $(\mathcal{C}_p)_{p\in[1,\infty]}$ yields a partition of the class of all bivariate copulas.

For a fixed $p$, let $S_3$ be the set $\mathscr{C}_p$ of $1$-$p$-Lipschitz copulas that are not $1$-$p'$-Lipschitz for any $p'>p$. By using the concepts presented in \cite{DeB10}, we will demonstrate that the set $S_3$ is maximal convex lineable.

Recall that a family $\{K_\alpha:\alpha\in[0,1]\}$ of infinite subsets of $\mathbb{N}$ is said to be {\it almost disjoint} if $K_\alpha\cap K_\beta$ is finite whenever $\alpha\neq \beta$.

Consider two one-parameter families of 2-copulas defined in \cite{DeB10}; on the one hand, $C_{(c)}:[0,1]^2\longrightarrow[0,1]$, given by
$$C_{(c)}(u,v)=\min\left\{u,v,\max\left\{\frac{u+v-c}{2-c},0\right\}\right\}$$
for all $(u,v)\in[0,1]^2$, with $c\in[0,1]$, whose support is illustrated in Figure \ref{fig:sop C_c}.
\begin{figure}[h!]
    \centering
    \includegraphics[width=0.4\linewidth]{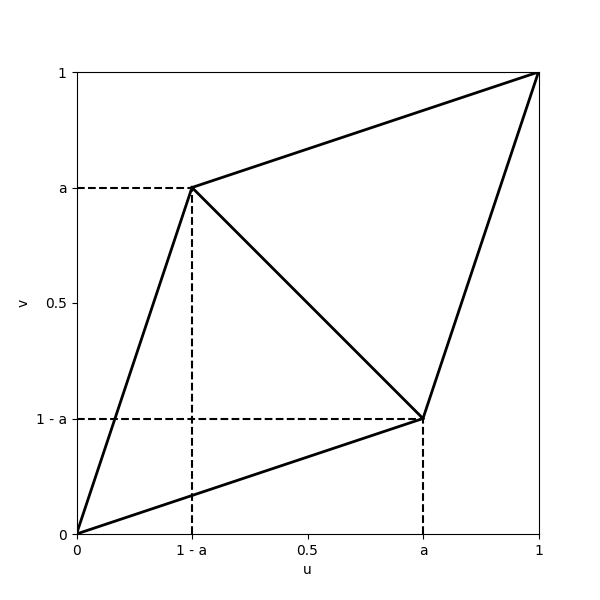}
    \caption{Support of the copula $C_{(c)}$.}
    \label{fig:sop C_c}
\end{figure}

On the other hand, $D_{(a)}:[0,1]^2\longrightarrow [0,1]$, given by
$$D_{(a)}(u,v)=\begin{cases}
    \min\{u,v,(1-a)(u+v)\}, & \text{ if } u+v\leq 1,\\
    \min\{u,v,a(u+v)+1-2a\}, &\text{ if } u+v\geq 1,
\end{cases}$$
for all $(u,v)\in[0,1]^2$ and $a \in[1/2,1]$, whose support is illustrated in 
Figure \ref{fig:sop D_a}.
\begin{figure}[h!]
    \centering
    \includegraphics[width=0.4\linewidth]{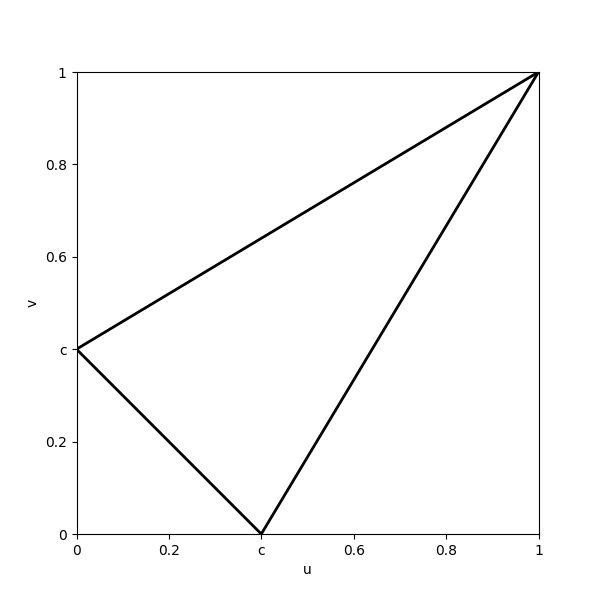}
    \caption{Support of the copula $D_{(a)}$.}
    \label{fig:sop D_a}
\end{figure}

Now we are in position to prove the following result:

\begin{theorem}\label{th:S3}
   The set $S_3$ is maximal convex lineable.
\end{theorem}

\begin{proof}
    Let $\{0,1\}^\mathbb{N}$ be the set of sequences whose components are $0$ or $1$. Let $\{K_\alpha:\alpha\in[0,1]\}$ an almost disjoint family of infinite subsets of $\mathbb{N}$ such that $K_\alpha(n)=1$ if $n\in K_\alpha$ and $K_\alpha(n)=0$ if $n\notin K_\alpha$. 

    Consider $C^{(\lambda)}=\lambda C_{(c)}+(1-\lambda)D_{(a)}$ for $\lambda\in[0,1]$. According to \cite[Example 4.2]{DeB10}, for suitable values of $a$ and $c$, the copulas $C_{(c)}, D_{(a)}$ and $C^{(\lambda)}$ are elements of $\mathscr{C}_p$.

    Now, we proceed to construct the ordinal sum $A_\alpha=\left(\left<\frac{1}{n+1},\frac{1}{n},C^{(K_\alpha)(n)}\right>\right)_{n\in\mathbb{N}}$. Suppose that these copulas are not linearly independent, then there exists a linear combination such that
    $$\sum_{i=1}^k\beta_iA_{\alpha_i}\equiv 0,$$
    with $\beta_i\in\mathbb{R}$ for all $i=1,\ldots,k$. There are infinitely many values of $n$ satisfying $K_{\alpha_1}(n)=1$ and $K_{\alpha_i}(n)=0$ for $i=2,\ldots,k$. It would then follow that $D_{(a)} = \beta\cdot C_{(c)}$ for a given $\beta$, which is impossible since the two copulas are distinct. Consequently, there exist $\mathfrak{c}$ linearly independent copulas. \cite[Proposition 4.4]{DeB10} guarantees that any convex combination of the copulas $A_\alpha$ belongs to $\mathscr{C_p}$. So, $S_3$ is maximal convex lineable.
\end{proof}

\subsection{Proper $n$-quasi-copulas}

Let $S_4$ be the set $\mathcal{Q}^n\setminus \mathcal{C}^n$. In this section, we investigate the internal algebraic and topological richness of the class $S_4$. Our primary objective is to determine the extent to which these set and that of proper 2-quasi-copulas with a given diagonal section are maximal spaceable convex lineable.

\begin{theorem}\label{th:S4}
  The set $S_4$ is maximal spaceable convex lineable.
\end{theorem}

\begin{proof}
We start studynig the case $n=2$. Given $C\in\mathcal{C}^2$, let $Q_C:[0,1]^2\longrightarrow [0,1]$ be the function defined by 
\begin{equation}\label{eq:QC}
Q_C(u,v)=\begin{cases}
    u+v-\displaystyle\frac{1}{3}C(3u-1,3v-1)-\frac{2}{3}, & (u,v)\in[1/3,2/3]^2,\\
    \max\left\{0,u-\displaystyle\frac{1}{3}, v-\frac{1}{3}, u+v-1\right\}, & (u,v)\notin[1/3,2/3]^2 \text{ and } |u-v|<1/3,\\
    \min\{u,v\}, & \text{ otherwise}.
\end{cases}
\end{equation}
$Q_C$ is indeed a proper quasi-copula, since $V_{Q_C}([1/3,2/3]^2)=-1/3$ (Figure \ref{fig:Supp Q_C} shows its support for $C=M^2$).
\begin{figure}[h!]
    \centering
    \includegraphics[width=0.4\linewidth]{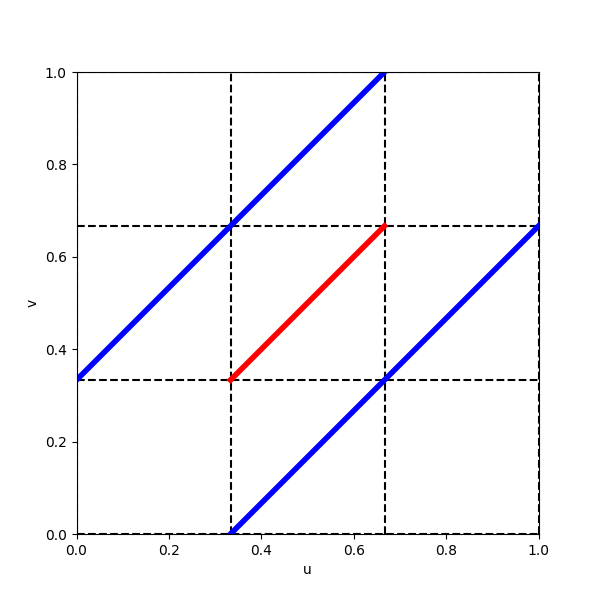}
    \caption{Support of the proper quasi-copula $Q_C$ given by \eqref{eq:QC} for $C=M^2$, with red indicating negative probability mass and blue indicating positive probability mass.}
    \label{fig:Supp Q_C}
\end{figure}

The mapping 
\[
\begin{split}
 \Phi:\mathcal{C}^2&\longrightarrow \mathcal{Q}^2 \\
 C&\longmapsto Q_C
\end{split}
\]
is continuous since, given $C_1,C_2\in\mathcal{C}^2$,

\begin{align*}
d_\infty(\Phi(C_1),\Phi(C_2))&=\sup_{(u,v)\in[0,1]^2}|Q_{C_1}(u,v)-Q_{C_2}(u,v)|\\
    &=\sup_{(u,v)\in[1/3,2/3]^2}\left|\frac{1}{3}\left(C_2(3u-1,3v-1)-C_1(3u-1,3v-1)\right)\right|\\
    &=\frac{1}{3}\sup_{(u,v)\in[0,1]^2}|C_2(u,v)-C_1(u,v)|\\
    &=\frac{1}{3}d_\infty(C_1,C_2).
\end{align*}

The metric space $(\mathcal{C}^2,d_\infty)$ is compact, so $\Phi(\mathcal{C}^2)$ is compact and therefore, closed. Note also that the image $\Phi(\mathcal{C}^2)$ is convex since it is straightforward that $\Phi(aC_1+(1-a)C_2)=a\Phi(C_1)+(1-a)\Phi(C_2)$ for all $a\in[0,1]$. Due to this last property, the mapping $\Phi$ sends linearly independent convex combinations of vectors to linearly independent convex combinations of vectors. Thus, $\mathcal{Q}^2\setminus\mathcal{C}^2$ is maximal spaceable convex lineable.

In order to prove the general case $n$, we define the continuous map 
\[
\begin{split}
 \Psi:\mathcal{Q}^2&\longrightarrow \mathcal{Q}^n \\
 Q&\longmapsto Q'
\end{split}
\]
where $Q'(u_1,u_2,\ldots,u_n)=Q(u_1,u_2)u_3\cdots u_n$. It is clear that $\Psi(\Phi(C))$ is a proper quasi-copula and the map $\Psi$ is continuous since
\begin{align*}
    d_\infty(\Psi(Q_1),\Psi(Q_2))&=\sup_{\uu\in[0,1]^n}|Q_1'(\uu)-Q_2'(\uu)|\\
    &=\sup_{\uu\in[0,1]^n}|Q_1(u_1,u_2)u_3\cdots u_n-Q_2(u_1,u_2)u_3\cdots u_n|\\
    &=\sup_{\uu\in[0,1]^n}|(Q_1(u_1,u_2)-Q_2(u_1,u_2)|\cdot|u_3\cdots u_n|\\
    &=\sup_{(u_1,u_2)\in[0,1]^2}|Q_1(u_1,u_2)-Q_2(u_1,u_2)|\\
    &=d_\infty(Q_1,Q_2).
\end{align*}
It is straightforward that $\Psi \left( \Phi \left( \mathcal{C}^{2}\right) \right)$ is convex and that $\Psi$ preserves convex combinations of linearly independent vectors. Since $\Psi \left( \Phi \left( \mathcal{C}^{2}\right) \right) \subset \mathcal{Q}
^{n}\setminus \mathcal{C}^{n}$, we have that $\mathcal{Q}^n\setminus\mathcal{C}^n$ is maximal spaceable convex lineable.
\end{proof}

\subsubsection{Proper $2$-quasi-copulas with given diagonal section}

We start this subsection with the definition of diagonal.

\begin{definition}
A mapping $\delta:[0,1]\longrightarrow[0,1]$ is said to be a \emph{diagonal} if it satisfies:
\begin{itemize}
    \item[i.] $\delta(1)=1$;
    \item[ii.] $\delta(t)\le t$ for all $t\in[0,1]$;
    \item[iii.] $\delta$ is increasing;
    \item[iv.] $|\delta(v)-\delta(u)|\le 2|v-u|$ for all $u,v\in[0,1]$.
\end{itemize}
\end{definition}

The diagonal section of any 2-quasi-copula $Q$ is given by $\delta_Q(t)=Q(t,t)$ for all $t\in[0,1]$.

Let $\mathcal{C}_\delta^2$ (respectively, $\mathcal{Q}_\delta^2$) denote the set of all 2-copulas (respectively, 2-quasi-copulas) having diagonal section $\delta$.
It is known (see \cite{Fre96,Ne96}) that for each diagonal $\delta$ there exists at least one 2-copula $C_\delta$ with that diagonal.

Let $S_4'$ denote the set of all proper 2-quasi-copulas sharing a fixed diagonal section $\delta$---different from $\delta(t)=t$---, i.e., $\mathcal{Q}_\delta^2 \setminus \mathcal{C}_\delta^2$. Then we have the following result.

\begin{theorem}
The set $S_4'$ is maximal spaceable convex lineable.
\end{theorem}

\begin{proof}
We use the construction given in \cite{Durante11}. Fix a diagonal $\delta$ such that there exists $t_0\in(0,1)$ with $\delta(t_0)\ne t_0$. Let $C\in\mathcal{C}_\delta^2$ be a 2-copula having that diagonal, and define the rectangle 
$R=[0,t_0]\times[t_0,1]$.

For any given proper quasi-copula $Q$ on $[0,1]^2$, define
\begin{equation}\label{eq:Ctilde}
\widetilde{C}(u,v) =
\begin{cases}
    V_C(R)\cdot Q(F(u),G(v)) + C(u,t_0), & (u,v)\in R,\\[3pt]
    C(u,v), & \text{otherwise},
\end{cases}
\end{equation}
where the rescaling functions $F,G:[0,1]\to[0,1]$ are given by
\[
F(u) = \frac{V_C([0,u]\times[t_0,1])}{ V_C(R)},
\qquad
G(v) = \frac{V_C([0,t_0]\times[t_0,v])}{ V_C(R)}.
\]
It follows from the properties of quasi-copulas (see \cite[Proposition 3]{Durante11}) that $\widetilde{C}$ is a quasi-copula sharing the same diagonal as $C$, and $\widetilde{C}$ fails to be a copula whenever $Q$ is proper. Hence, $\widetilde{C}\in \mathcal{Q}_\delta^2 \setminus \mathcal{C}_\delta^2$.

Let $L=\Psi \left( \Phi \left( \mathcal{C}^{2}\right) \right) \subset \mathcal{Q}
^{n}\setminus \mathcal{C}^{n}$---recall the proof of Theorem \ref{th:S4}---, and let $\Omega$ be the map defined as follows
\[
\begin{split}
 \Omega:L&\longrightarrow \mathcal{Q}^2 \\
 Q&\longmapsto \widetilde{C}_Q
\end{split}
\]
This map is continuous, ideed it is Lipschitz, because 
\begin{align*}
    d_\infty(\Omega(Q_1),\Omega(Q_2))&=\sup_{(u,v)\in[0,1]^2}|\widetilde{C}_{Q_1}(u,v)-\widetilde{C}_{Q_2}(u,v)|\\
    &=\sup_{(u,v)\in R}|V_C(R)Q_1(F(u),G(v))-C(u,t_0)-V_C(R)Q_2(F(u),G(v))+C(u,t_0)|\\
    &=\sup_{(u,v)\in R}|V_C(R)(Q_1(F(u),G(v))-Q_2(F(u),G(v))|\\
    &=V_C(R)\sup_{(u,v)\in R}|Q_1(F(u),G(v))-Q_2(F(u),G(v))|\\
    &=V_C(R)\sup_{(u,v)\in [0,1]^2}|Q_1(u,v)-Q_2(u,v)|\\
    &=V_C(R)\cdot d_\infty(Q_1,Q_2).
\end{align*}
So $\Omega(L)$ is closed. Now, we show that $\Omega(\lambda Q_1+(1-\lambda)Q_2 )=\lambda\Omega(Q_1)+(1-\lambda)\Omega(Q_2)$ for $\lambda\in[0,1]$; but it is clear because if $(u,v)\notin R$, then $\Omega(Q_1)(u,v)=\Omega(Q_2)(u,v)=C(u,v)$, and if $(u,v)\in R$, then  
\begin{align*}
    \Omega(\lambda Q_1+(1-\lambda)Q_2)(u,v)&=V_C(R)(\lambda Q_1+(1-\lambda)Q_2)(F(u),G(v))+C(u,t_0)\\
    &=\lambda V_C(R)Q_1(F(u),G(v))+(1-\lambda)V_C(R)Q_2(F(u),G(v))+ (\lambda+1-\lambda)C(u,t_0)\\
    &=\lambda\Omega(Q_1)(u,v)+(1-\lambda)\Omega(Q_2)(u,v).
\end{align*}
$\Omega$ also preserves linearly independent vectors, so  $\mathcal{Q}_\delta^2\setminus\mathcal{C}_\delta^2$ is maximal spaceable convex lineable.
\end{proof}

\subsection{Sequences of copulas}

The study of the space of copulas $\mathcal{C}^2$ often relies on various topologies. In general settings, these topologies are not equivalent, leading to the creation of sets composed of sequences that converge under one metric but diverge under another.

\subsubsection{Sequences of copulas converging in $d_\infty$ but not converging in $D_1$}

We recall that the standard {\it uniform metric} $d_\infty$ on $\mathcal{C}^2$ is defined by
\[
d_\infty(C_1, C_2) := \max_{(x,y) \in [0,1]^2} |C_1(x, y) - C_2(x, y)|. \tag{2.1}
\]
It is well known that the metric space $(\mathcal{C}, d_\infty)$ is compact and that pointwise and uniform convergence of a sequence of copulas $(C_n)_{n\in\mathbb{N}}$ are equivalent (see \cite{Durante2016book}).

On the other hand, for the metric $D_1$ we need some preliminary concepts.

\begin{definition}
    A Markov Kernel from $\mathbb{R}$ to $\mathcal{B}(\mathbb{R})$, the Borel $\sigma$-field on $\mathbb{R}$, is a mapping $K:\mathbb{R}\times\mathcal{B}(\mathbb{R})\longrightarrow[0,1]$ such that $x\mapsto K(x,B)$ is measurable for every fixed $B\in\mathcal{B}(\mathbb{R})$ and $B\mapsto K(x,B)$ is a probability measure for every fixed $x\in\mathbb{R}$. 
\end{definition}

Given two random variables $X$ and $Y$ in a probability space $(\Omega,\mathcal{A},\mathbb{P})$, a Markov kernel, $K$, is called regular conditional distribution of $Y$ given $X$ if for every $B\in\mathcal{B}(\mathbb{R})$
$$K(X(\omega),B)=\mathbb{E}(\mathds{1}_B\circ Y|X)(\omega)$$
holds $\mathbb{P}$-a.e. If $C\in\mathcal{C}^2$ is the copula associated with $X$ and $Y$, $E,F\in\mathcal{B}([0,1])$, then 
$$\int_FK_C(u,E)d\lambda(u)=\mu_C(F\times E),$$
where $K_C$ denotes the regular conditional distribution of $Y$ given $X$ and $\mu_C$ is the probability measure associated with $C$.

In $\mathcal{C}^2$, the metric $D_1$ is defined by 
$$D_1(C_1,C_2):=\iint_{[0,1]^2}|K_{C_1}(u,[0,v])-K_{C_2}(u,[0,v])|{\rm d}\lambda(u){\rm d}\lambda(v).$$
The metric space $(\mathcal{C}^2,D_1)$ is complete and separable (see \cite{Tru11}). 

Let $S_5$ be the family of sequences of copulas that converge under the $d_\infty$ metric but fail to converge under the $D_1$ metric (an example of a sequence of copulas in $S_5$ can be found in \cite[Example 5.1]{Ka18}).
\begin{remark}
    We recall that $(\mathcal{C}^2)^\mathbb{N}$---the set of all sequences of 2-copulas---is a subset of the vector space of the sequences of bounded functions on $[0,1]^\mathbb{N}$.
\end{remark}
In order to generalize the classical concept of shuffles of Min, in \cite{Durante10} two fundamental notions were introduced: the \emph{S-structure} and the \emph{S-copula}. An S-structure provides a systematic way to partition the unit interval for each marginal dimension, ensuring that the resulting multidimensional grid preserves measure consistency across coordinates. Based on such a structure, an S-copula (or shuffling copula) is constructed by redistributing, within each block of the partition, the probability mass of a family of copulas through suitable affine transformations. This framework encompasses both the multivariate shuffle of Min and the ordinal sum construction, and plays a central role in the approximation of arbitrary copulas under the $L^\infty$ norm.

\begin{definition}[S-structure]
Let $d \ge 2$. 
An \emph{S-structure} (or \emph{shuffling structure}) is a family 
$S = (J^i)_{i=1}^d$, where each $J^i = (J^i_n)_{n \in N}$ is a system of closed, non-empty subintervals of $[0,1]$, 
satisfying the following conditions:
\begin{itemize}
    \item[i)] $N$ is a finite or countable index set, i.e., $N = \{0,1,\ldots,\tilde{n}\}$ or $N = \mathbb{Z}_+$;
    \item[ii)] for every $i \in \{1,\ldots,d\}$ and all $n,m \in N$ with $n \neq m$, the intervals $J^i_n$ and $J^i_m$ have at most one endpoint in common;
    \item[iii)] for every $i \in \{1,\ldots,d\}$,
    $$\sum_{n \in N} \lambda(J^i_n) = 1,$$
    where $\lambda$ denotes the Lebesgue measure on $[0,1]$;
    \item[iv)] for every $n \in N$, all intervals $J^1_n, J^2_n, \ldots, J^d_n$ have the same length, i.e.,
    $$\lambda(J^1_n) = \lambda(J^2_n) = \cdots = \lambda(J^d_n).$$
\end{itemize}
The collection of all S-structures based on the index set $N$ is denoted by $\mathscr{S}_N$, and 
$\mathscr{S} = \bigcup_N \mathscr{S}_N$.
\end{definition}

\begin{definition}[S-copula]
An S-copula (short for shuffling copula) is a copula whose induced measure can be represented as a shuffling measure that is $d$-fold stochastic. Formally, an S-copula 
$$C = \langle (J^i)_{i=1}^d, (C_n)_{n \in N} \rangle$$
is obtained by partitioning the unit cube $[0,1]^d$ into blocks 
$J^1_n \times \cdots \times J^d_n$ (an \emph{S-structure}) and redistributing, within each block, 
the probability mass of a family of copulas $(C_n)_{n \in N}$ through suitable affine transformations. 
The resulting measure remains $d$-fold stochastic and therefore defines a valid copula.
\end{definition}
A particular case in which we are interested is obtained when there exists a copula $C$ such that $C_n=C$ for any $n\in N$. In this case, it is possible to define for any $S\in\mathscr{S}_N$ a map $\widetilde{\psi}_{(S,\cdot)}:\mathcal{C}^2\longrightarrow \mathcal{C}^2$, given by $\widetilde{\psi}_{(S,\cdot)}(C)=\left<S,(C_n=C)_{n\in N}\right>$.

We are now in position to prove the following result.

\begin{theorem}
   The set $S_5$ is maximal convex lineable.
\end{theorem}

\begin{proof}
For each integer $N\geq 2$, consider the equal-length partition of $[0,1]$
    $$I_n^1=\left[\frac{n-1}{N^2},\frac{n}{N^2}\right], \hspace{1cm}n=1,2,...,N^2,$$
    $$I_{N(j-1)+k}^2=\left[\frac{N(k-1)+j-1}{N^2},\frac{N(k-1)+j}{N^2}\right]\text{ for } k,j=1,2,...,N$$
(see \cite{Mi1992,Ne06}), and the corresponding blocks $I_n^1\times I_n^2$. This yields an $S$-structure $S_N'\in\mathcal{S}$ (see \cite{Durante10}).

We define $\Theta_N:\mathcal{C}^2\longrightarrow\mathcal{C}^2$ so that for a given copula $C\in\mathcal{C}^2$, $\Theta_N(C)$ is defined by placing on each block $I_n^1\times I_n^2$ an affine-rescaled copy of $C$ for $(u,v)\in I_n^1\times I_n^2$, with mass $1/N^2$. This is the $S$-copula associated to $S_N$ with constant block family $C$ (see \cite[Proposition 2.2]{Durante10}).
    
    Regarding the argument in \cite{Mi1992,Ne06}, it holds that
$$\lim_{N\to\infty}d_\infty(\Theta_N(C),\Pi^2)=0.$$
Thus $\Theta_N(C)$ converges uniformly to the independence copula $\Pi^2$.

On the other hand, we consider the equal-length partition of $[0,1]$
$$J_n^1=J_n^2=\left[\frac{n-1}{N^2},\frac{n}{N^2}\right],\hspace{1cm} n=1,2,...,N^2,$$
and the corresponding diagonal blocks $J_n^1\times J_n^2$. This yields an $S$-structure $S_N\in\mathcal{S}$ (see \cite{Durante10}).
    
Now, we fixed a copula $C\in\mathcal{C}^2$, $C\neq \Pi^2$, and we define $\Psi_N(C)$ by placing on each diagonal block $J_n^1\times J_n^2$ an affine-rescaled copy of $C$; that is, for $(u,v)\in J_n^1\times J_n^2$,
$$\Psi_N(C)(u,v)=\frac{n-1}{N^2}+\frac{1}{N^2}C(N^2u-(n-1),N^2v-(n-1)).$$
Since $\Pi^2$ is invariant under the action of generalized shuffles, it holds that $D_1(\Psi_N(C),\Pi^2)=D
    _1(\Theta_N(C),\Pi^2)$. Therefore, we will study $D_1(\Psi_N(C),\Pi^2):$
    \begin{align*}
D_1(\Psi_N(C),\Pi^2)&=\int\int_{[0,1]^2}|K_{\Psi_N(C)}(u,[0,v])-v|dvdu\\
&\geq\sum_{n=1}^{N^2}\left(\int\int_{\left[\frac{n-1}{N^2},\frac{n}{N^2}\right]\times\left[0,\frac{n-1}{N^2}\right]}v\hspace{0.1cm} dudv+\int\int_{ \left[\frac{n-1}{N^2},\frac{n}{N^2}\right]\times\left[\frac{n}{N^2},1\right]}(1-v) dudv\right)\\
        &=\frac{1}{N^2}\sum_{n=1}^{N^2}\left(\int_{\left[0,\frac{n-1}{N^2}\right]}v\hspace{0.1cm} dv+\int_{\left[\frac{n}{N^2},1\right]}(1-v) dv\right)\\
        &=\frac{2N^4-3N^2+1}{6N^4}.
    \end{align*}
    
        Additionally, 
        \begin{align*}
            D_1(\Psi_N(C),\Pi^2)=)&=\int\int_{[0,1]^2}|K_{\Psi_N(C)}(u,[0,v])-v|dvdu\\
        &\leq\sum_{n=1}^{N^2}\left(\int\int_{\left[\frac{n-1}{N^2},\frac{n}{N^2}\right]\times\left[0,\frac{n-1}{N^2}\right]}v\hspace{0.1cm} dudv+\int\int_{ \left[\frac{n-1}{N^2},\frac{n}{N^2}\right]\times\left[\frac{n}{N^2},1\right]}(1-v) dudv\right.\\
        &\left.+\int\int_{ \left[\frac{n-1}{N^2},\frac{n}{N^2}\right]^2}1\hspace{0.1cm} dudv\right)\\
        &=\frac{1}{N^2}\sum_{n=1}^{N^2}\left(\int_{\left[0,\frac{n-1}{N^2}\right]}v\hspace{0.1cm} dv+\int_{\left[\frac{n}{N^2},1\right]}(1-v)dv+\frac{1}{N^2}\right)\\
        &=\frac{2N^4+3N^2+1}{6N^4}.
        \end{align*}
    Therefore, 
    $$\lim_{N\to\infty}D_1(\Psi_N(C),\Pi^2)=\frac{1}{3}.$$
    Then,
    $$\lim_{N\to\infty}D_1(\Theta_N(C),\Pi^2)=\frac{1}{3}.$$
    Given that the mapping assigns to each copula $C$ an affine-rescaled copy of it in every block $I^1_n\times I_n^2$, it is evident that $\Theta_N(\lambda C_1+(1-\lambda)C_2)=\lambda\Theta_N(C_1)+(1-\lambda)\Theta_N(C_2)$, for every $\lambda\in[0,1]$ and then it preserves linearly independent vectors. Thus, $S_5$ is maximal convex lineable.
\end{proof}

\subsubsection{Sequences of copulas converging in $D_1$ but not converging in total variation}

Let $(\Omega, \mathcal{A})$ be a measurable space and let $\mu$ be a signed measure on $(\Omega, \mathcal{A})$. The \textit{total variation norm} of $\mu$, denoted by $\Vert \mu \Vert_{\rm var}$, is defined as the total mass of the absolute value measure $|\mu|$, which is given by:
\[
\Vert \mu \Vert_{\rm var} := |\mu|(\Omega) = \sup \left\{ \sum_{i=1}^\infty |\mu(A_i)| \right\},
\]
where the supremum is taken over all countable partitions $\{A_i\}_{i=1}^\infty$ of $\Omega$ into disjoint measurable sets $A_i \in \mathcal{A}$. Alternatively, by the Hahn Decomposition Theorem, $\Omega$ can be partitioned into a positive set $P$ and a negative set $N$ such that $\mu = \mu^+ - \mu^-$. The (total) variation norm is then the sum of the norms of the positive and negative parts:
\[
\Vert \mu \Vert_{\rm var} = \mu^+(\Omega) + \mu^-(\Omega) = \mu(P) - \mu(N).
\]
(see, e.g., \cite{Royden1988}).

A sequence of 2-copulas $\{C_k\}_{k\in\mathbb{N}}$ is said to converge to a limit 2-copula $C$ in (total) variation if the distance between their associated measures, measured by the $L_1$ norm of their densities---when they exist---, tends to zero:
$$
\lim_{k \to \infty} \Vert \mu_{C_k} - \mu_{C} \Vert_{\rm var} = 0.
$$
Probabilistically, this means that the sequence $\{C_k\}$ becomes indistinguishable from the limit $C$ in all respects, as the difference between their probability assignments over any measurable set vanishes.

Convergence in (total) variation is strictly stronger than both the uniform metric $d_{\infty}$ and the metric $D_1$. Let $S_6$ be the family of sequences of copulas that converge under the $D_1$ metric but fail to converge in total variation. An example of a sequence of copulas in $S_6$ is the following:

\begin{example}\label{ex:variation}
Let \(D\) denote the independence copula \(\Pi(u,v)=uv\) (we will use its kernel \(K(x,[0,y])=y\)).
For each integer \(n\ge1\) define the 2-copulas
\begin{equation}\label{eq:variation}
C_n(u,v)
= uv + \frac{(1 - \cos(2\pi u))(1 - \cos(2\pi n v))}{8\pi^2 n},
\quad (u,v)\in[0,1]^2.
\end{equation}
with densities on \([0,1]^2\)
\begin{equation}\label{eq:density}
c_n(u,v)=1+\frac{1}{2}\,\sin(2\pi u)\,\sin(2\pi n v),\quad (u,v)\in[0,1]^2.
\end{equation}
Set \(C_n\) to be the copula whose density w.r.t.\ Lebesgue measure is \(c_n\) (it is straightforward to check
that \(c_n\ge 1/2>0\) so it is a bona fide density and, because the \(y\)-integral of \(\sin(2\pi n y)\) is \(0\),
both marginals are uniform and \(C_n\) is a copula).

Let \(C\) be the independence copula (with kernel \(K(x,[0,y])=y\)). We show:
\begin{enumerate}
  \item \(D_1(C_n,C)\to0\) as \(n\to\infty\);
  \item the associated measures do not converge in total variation: \(\|\mu_{C_n}-\mu_C\|_{\mathrm{var}}\not\to0\).
\end{enumerate}

For an absolutely continuous copula with density \(c(x,y)\) the Markov kernel satisfies
\[
K(x,[0,y])=\int_0^y c(x,t)\,{\rm d}t.
\]
Hence for our \(c_n\),
\[
K_n(x,[0,y])-K(x,[0,y])
= \int_0^y \frac{1}{2}\sin(2\pi x)\sin(2\pi n t)\,{\rm d}t
= \frac{\sin(2\pi x)}{4\pi n}\big(1-\cos(2\pi n y)\big).
\]
Therefore, using the definition
\[
D_1(C_n,C)=\iint_{[0,1]^2}\big|K_n(x,[0,y])-K(x,[0,y])\big|\,{\rm d}x\,{\rm d}y,
\]
we obtain the uniform bound
\[
\big|K_n(x,[0,y])-K(x,[0,y])\big|
\le \frac{|\sin(2\pi x)|}{2\pi n},
\]
and hence
\[
D_1(C_n,C)\le \frac{1}{2\pi n}\int_0^1 |\sin(2\pi x)|\,{\rm d}x
= \frac{1}{2\pi n}\cdot\frac{2}{\pi}
= \frac{1}{\pi^2 n}\longrightarrow 0.
\]
Thus \(C_n\to C\) in the metric \(D_1\).

On the other side, the total variation distance between the two probability measures \(\mu_{C_n}\) and \(\mu_C\) equals
the \(L_1\)-distance of their densities:
\[
\|\mu_{C_n}-\mu_C\|_{\mathrm{var}}
= \int_{[0,1]^2} |c_n(x,y)-1|\,{\rm d}x\,{\rm d}y
= \frac{1}{2}\int_0^1 |\sin(2\pi x)|\,{\rm d}x\;\cdot\;\int_0^1 |\sin(2\pi n y)|\,{\rm d}y.
\]
Each factor equals \(\int_0^1 |\sin(2\pi t)|\,{\rm d}t = 2/\pi\), so
\[
\|\mu_{C_n}-\mu_C\|_{\mathrm{var}}
= \frac{1}{2}\cdot\frac{2}{\pi}\cdot\frac{2}{\pi}
= \frac{2}{\pi^2} > 0,
\]
independently of \(n\). In particular, the variation distance does \emph{not} tend to \(0\).
\end{example}

The set $S_6$ possesses the maximal convex lineable structure as the following result shows.

\begin{theorem}
   The set $S_6$ is maximal convex lineable.
\end{theorem}

\begin{proof}
    In order to prove the statement of the theorem, we will use the sequence $\{C_n\}_{n\in\mathbb{N}}$ in Example \ref{ex:variation}. Let $\{0,1\}^\mathbb{N}$ be the set of sequences whose components are $0$ and $1$. Let $\{K_\alpha:\alpha\in[0,1]\}$ be an almost disjoint family of infinite subsets of $\mathbb{N}$. We consider the sequence $\{C_{K_\alpha,n}\}_n\in\mathbb{N}$ such that 
    $$C_{K_\alpha,n}=\begin{cases}
        C_n, & \text{ if }  n\in K_\alpha,\\
        \Pi^2, & \text{ if } n\notin K_\alpha.
    \end{cases}$$
    Now, we show that these copulas are linearly independent: suppose that these copulas are not linearly independent, so there exists $\beta_i\in\mathbb{R}$ for $i=1,\ldots,k$ , with some $\beta_i\neq 0$ such that
    $$\sum_{i=1}^k\beta_i\{C_{K_{\alpha_i},n}\}\equiv 0.$$
    There are infinitely many values of $n$ satisfying $K_{\alpha_1}(n)=1$ and $K_{\alpha_i}(n)=0$ for $i=2,...,k$. It would then follow that by equating the terms of the sequence term by term, we obtain $\beta_1 C_n=\beta \Pi^2$, but these two copulas are not equal for every $(u,v)\in[0,1]^2$. So the sequences $\{C_{K_\alpha,n}\}_n\in\mathbb{N}$ are linearly independent and there are $\mathfrak{c}$ of them, since $\alpha\in[0,1]$.

    From Example \ref{ex:variation}, it holds that
    $$\lim_{n\to\infty}D_1(C_{K_{\alpha},n},\Pi^2)=0,$$
    since just for a finite subset of $\mathbb{N}$, $C_{K_{\alpha},n}=\Pi^2$. It is also satisfied that convex combinations of these copulas converge to $\Pi^2$ in the metric $D_1$. As $K_{\alpha_1}\cap K_{\alpha_2}=\emptyset$, then $\lambda C_{K_{\alpha_1},n}+(1- \lambda)C_{K_{\alpha_2},n}=C_n$, for $\lambda\in[0,1]$, whenever $n\to\infty$. 

    On the other hand, applying similar arguments as above
    $$\lim_{n\to\infty}\|\mu_{C_{K_\alpha},n}-\mu_{\Pi^2}\|_{\mathrm{var}}=\lim_{n\to\infty}\|\mu_{C_n}-\mu_{\Pi^2}\|_\mathrm{var}>0.$$
    By the same reasoning as before, we deduce that convex combinations of these copulas do not converge to $\Pi^2$ either. So $S_6$ is maximal convex lineable.
\end{proof}

\section{Conclusion}\label{sec:conc}

In this work, we have studied the algebraic and topological largeness of several families of copulas and quasi-copulas through the notions of convex lineability and convex spaceability. Our results show that many classical and recently investigated subclasses of $n$-copulas exhibit a surprisingly rich internal structure. We proved that a number of these families are maximal convex lineable, meaning that they contain linearly independent sets of cardinality $\mathfrak{c}$ whose entire convex hull remains inside the class. This holds, for instance, for copulas with prescribed fractal support, copulas with bounded Lipschitz-type constraints, and sequences of copulas displaying different convergence behaviors depending on the chosen topology. For other important classes, convex lineability can be strengthened to maximal spaceable convex lineability. This is the case for the family of asymmetric copulas, copulas with maximal asymmetry, proper quasi-copulas, and proper 2-quasi-copulas with a fixed diagonal section. In these situations we identified closed, convex, infinite linearly independent subsets, thereby providing a significantly deeper structural understanding of the underlying families.

On the other hand, in certain natural settings convex spaceability remains unknown. This happens, for example, for copulas with fractal support or for 1-$p$-Lipschitz copulas that are not 1-$p'$-Lipschitz for any $p'>p$. In these contexts our constructions guarantee maximal convex lineability, but the lack of closure of the linear span prevents us from asserting spaceability. These examples highlight subtle obstructions caused by ordinal-sum mechanisms, mutual singularity of induced measures, or discretized block constructions.

The results obtained here open several lines of research. Among the most relevant are: (i) determining whether the families for which convex spaceability remains open actually contain closed infinite linearly independent subsets; (ii) characterizing the stability of convex lineable and spaceable structures under operations such as composition or iteration of shuffles; (iii) study the convex lineability and spaceability for other sets of copulas, e.g., copulas with hairpin supports \cite{Chamizo2021}; and (iv) extending the analysis to higher-dimensional dependence structures and alternative metrics.

\bigskip

\noindent
{\bf Acknowledgments}
\medskip

The first author is supported by the CDTIME of the University of Almería (Spain).

\end{document}